\documentclass[10pt]{amsart}
\usepackage{amssymb, tikz, mathrsfs}
\usepackage[all]{xy}

\title{Canonical heights for correspondences}
\author{Patrick Ingram}
\date{\today}
\address{Fort Collins, USA}

\newcommand{\QQ}{\mathbb{Q}}

\newcommand{\RR}{\mathbb{R}}

\newcommand{\PP}{\mathbb{P}}

\newcommand{\Div}{\operatorname{Div}}

\newcommand{\path}{\mathscr{P}}

\newcommand{\cadiv}{\operatorname{Div}}
\newcommand{\augdiv}{\operatorname{ADiv}}
\newcommand{\nediv}{\operatorname{NDiv}}

\renewcommand{\phi}{\varphi}
\renewcommand{\epsilon}{\varepsilon}

\newtheorem{theorem}{Theorem}[section]
\newtheorem{lemma}[theorem]{Lemma}
\newtheorem{proposition}[theorem]{Proposition}

\newtheorem{corollary}[theorem]{Corollary}

\theoremstyle{remark}
\newtheorem{remark}[theorem]{Remark}
\newtheorem{question}[theorem]{Question}
\newtheorem{example}[theorem]{Example}

\theoremstyle{definition}
\newtheorem{definition}[theorem]{Definition}

\begin{document}
\begin{abstract}
The canonical height associated to a polarized endomporhism of a projective variety, constructed by Call and Silverman and generalizing the N\'{e}ron-Tate height on a polarized Abelian variety, plays an important role in the arithmetic theory of dynamical systems. We generalize this construction to polarized correspondences, prove various fundamental properties, and show how the global canonical height decomposes as an integral of a local height over the space of absolute values on the algebraic closure of the field of definition.
\end{abstract}
\maketitle

\section{Introduction}

\subsection{Dynamics of correspondences}
Let $X$ be an algebraic variety defined over some field $K$, and let $C$ be a \emph{correspondence} on $X$, by which we mean a subvariety $C\subseteq X\times X$ such that the coordinate projections are both finite and surjective.
 We may think of the points $X(\overline{K})$ as a set of vertices, with the points $C(\overline{K})$  defining directed edges between those vertices, and inquire about the structure of the resulting graph. For instance, if $f:X\to X$ is a finite morphism, then the graph of $f$ defines a correspondence on $X$, and the paths in the aforementioned directed graph correspond to orbits in the sense of dynamics. Thus the study of paths in the directed graph associated to a correspondence is a natural generalization of the study of dynamical systems.
 
  The purpose of this note is to associate a canonical height to certain correspondences over number fields in a way which generalizes the Call-Silverman canonical height associated to a polarized algebraic dynamical system~\cite{callsilv}. We note that Autissier~\cite{aut1, aut2} has previously considered canonical heights associated to correspondences, but our construction differs significantly from his; the relation between the two is discussed in Remark~\ref{rem:autissier} below.

In the study of the iteration of a morphism $f:X\to X$, orbits are defined entirely in terms of their initial points (we shall henceforth refer to this as the \emph{deterministic} case), and so many statements about orbits end up being quantified over the set of points $X(K)$ or $X(\overline{K})$, rather than over the set of orbits \emph{per se}. In the study of correspondences this becomes inappropriate, and so we define below a \emph{path space} $\pi:\path\to X$ for a given correspondence, parametrizing paths through the aforementioned directed graph, with initial vertex indicated by $\pi$. We will also construct a \emph{shift map} $\sigma:\path\to\path$ which corresponds to stepping forward once along the path (or, equivalently, forgetting the first vertex of a path). These objects can be constructed in the category of $K$-schemes, but generally not in the category of varieties over $K$, a distinction which becomes important in our discussion of local and global heights. We define $\path$ and $\sigma$ explicitly in Section~\ref{sec:geom}, but for now the reader in encouraged to simply think of $\path(K)$ as being the set of paths whose points all sit in $X(K)$.

\subsection{The global canonical height}
When $K$ is a number field, 
one of the central tools in the study of polarized algebraic dynamical systems over $K$ is the canonical height $\hat{h}_{X, f, L}:X(\overline{K})\to \RR$ of Call and Silverman~\cite{callsilv},  a construction extending classical work of N\'{e}ron and Tate in the context of abelian varieties. Here $X$ is a projective variety, $f:X\to X$ is an endomorphism, and $L$ is an ample divisor on $X$ satisfying $f^*L\sim \alpha L$ in $\Div(X)\otimes \RR$, for some real $\alpha>1$.
Our first result is the existence of a generalization of the Call-Silverman canonical height in the context of dynamical correspondences. 

Let $C$ be a correspondence on $X$ as above, and let $x, y:C\to X$ be the coordinate projections.
We will say that this correspondence is \emph{polarized} if and only if there exists an ample Cartier divisor $L\in \Div(X)\otimes\RR$ and a real number $\alpha>1$ such that $y^* L$ is linearly equivalent to $\alpha x^*L$. It is easy to check that in the case where $C$ is the graph of a morphism $f:X\to X$, the correspondence is polarized just in case the algebraic dynamical system is polarized in the usual sense.

\begin{theorem}
\label{th:canheight}
Given a correspondence $C$ on $X$, defined over a number field $K$ and polarized by $L\in\Div(X)\otimes\RR$ and $\alpha>1$, there exists a function $\hat{h}_{X, C, L}:\path(\overline{K})\to \RR$ such that for all $P\in \path(\overline{K})$
\[\hat{h}_{X, C, L}(\sigma(P))=\alpha\hat{h}_{X, C, L}(P)\]
and
\[\hat{h}_{X, C, L}(P)=h_{X, L}(\pi(P))+O(1).\]
\end{theorem}

As one might suspect, in the case where $C$ is the graph of a morphism $f:X\to X$, we simply recover the usual canonical height (and in this case $\pi:\path\to X$ turns out to be an isomorphism).

\begin{remark}\label{rem:autissier}
Autissier~\cite{aut2} constructs a canonical height associated to a correspondence under certain conditions, and we reflect here on the relation between the constructions. Specifically, suppose that we are in the context of Theorem~\ref{th:canheight} with the additional hypothesis that $\alpha\in\QQ$. Then Autissier constructs a real-valued function $h'_{X, C, L}$ on 1-cycles of $X$ with the property that
\[h'_{X, C, L}(C_* D)=\deg(x)h'_{X, C, L}(D)+\deg(x)\deg(D_K)h_0,\]
where $C_*=y_*x^*$ and $h_0$ is the \emph{charge} of the correspondence, defined in \cite{aut2}.

Autissier's heigh is defined on (1-cycles on) $X(\overline{K})$, while the height constructed in Theorem~\ref{th:canheight} is a function on $\path(\overline{K})$. Note, however, that for any $a\in X(\overline{K})$, the fibre $\pi^{-1}(a)\subseteq \path(\overline{K})$ naturally admits the structure of a probability space. One may construct this measure as the inverse limit of the uniform distribution on paths of length $n$ (weighting the paths appropriately when there is ramification), and it is easy to show that $\hat{h}_{X, C, L}$ becomes a measurable function on this space, that is, a random variable. Let $\mathbb{E}\hat{h}_{X, C, L}(a)$ denote the expected value of $\hat{h}_{X, C, L}(P)$ for $\pi(P)=a$, in other words
\[\mathbb{E}\hat{h}_{X, C, L}(a)=\int_{\pi(P)=a}\hat{h}_{X, C, L}(P)d\mu(P),\]
where $\mu$ is the distribution defined above. It is easy to confirm from Theorem~\ref{th:canheight} that (extending this function linearly to 1-cycles)
\[\mathbb{E}\hat{h}_{X, C, L}(C_* a)=\deg(x)\mathbb{E}\hat{h}_{X, C, L}(a).\]

Since
\[\mathbb{E}\hat{h}_{X, C, L}(a)=h_{X, L}(a)+O(1)=h'_{X, C, L}(a)+O(1),\]
we see that in the special case $h_0=0$ and $\deg(x)>1$, we must actually have $h'_{X, C, L}=\mathbb{E}\hat{h}_{X, C, L}$. In other words, at least in some cases, Autissier's canonical height is the average value of the canonical height constructed in this note.
\end{remark}

One of the useful features of the canonical height in deterministic arithmetic dynamics is that its vanishing identifies the preperiodic orbits.  A path $P\in\path(\overline{K})$ will be called
 \emph{repetitive} if and only if $\pi\circ\sigma^n(P)=\pi\circ\sigma^m(P)$ for some $n\neq m$ (in other words, if and only if the path passes through some vertex twice). Note that, unlike in the case of deterministic dynamics, repetitive paths need not repeat more than once.
\begin{corollary}\label{cor:finitenessforrepetitive}
The set of points $a\in X(\overline{K})$ such that $\pi^{-1}(a)\subseteq \path(\overline{K})$ contains a repetitive path is a set of bounded height. In particular, there are only finitely many $a\in X(K)$ which are the initial vertex of some repetitive path over $\overline{K}$.
\end{corollary}

Note that it is quite important, in the statement of the corollary, that we consider paths defined over the algebraic closure. Unlike in the deterministic case, paths are not generally defined over the field of definition of the initial point, and indeed are typically not even defined over a finite extension of this field. Furthermore, the finiteness in the above statement must apply to initial points, rather than paths, since it is easy to construct examples wherein $\pi^{-1}(a)$ contains uncountably many repetitive paths for a single $a\in X(K)$. In order to elucidate the situation, we present an example below. A path $P\in\path(\overline{K})$ is said to be \emph{finitely supported} if and only if the set of points of the form $\pi\circ\sigma^n(P)\in X(\overline{K})$ is finite as $n$ varies, in other words, if and only if there are only finitely many distinct points of $X$ occurring as vertices in the path. Note that this is a stronger condition than being repetitive, and is strong enough to immediately imply $\hat{h}_{X, C, L}(P)=0$.

\begin{example}\label{ex:binary}
Consider the correspondence defined on $\PP^1$ by the closure in $\PP^1\times\PP^1$  of the curve
\[C:y^2=x^3-x+1,\]
with $x$ and $y$ corresponding to the  coordinate projections.
Here $\path(\QQ)$ turns out to be uncountable, but it is also true that every path in $\path(\QQ)$ is finitely supported. There are two ways to see this latter claim.

First of all, the curve
\[\left\{\begin{array}{c}y^2=x^3-x+1\\z^2=y^3-y+1\end{array}\right\}\]
parametrizing paths of length 3 under this correspondence is a curve of genus 6. Hence, by Faltings' Theorem, there are only finitely many points $a\in \PP^1(\QQ)$ from which we may build a path of length 3, let alone an infinite path. 

From another view, suppose that $K/\QQ$ is the field of fractions of a DVR with valuation $v$, and that $(a, b)\in C(K)$, with $v(a)<0$. Then we have $v(b)=\frac{3}{2}v(a)$ and so, in particular, there can be no path $P\in\path(K)$ with initial vertex $\pi(P)=a$ (such a path would contradict $v$ being discrete). Turning this around, if $P\in \path(\QQ)$, then each edge $a\to b$ of $P$ must correspond to an integral point $(a, b)\in C(\QQ)$. It is an easy exercise to show that the integral points on $C\subseteq \PP^1\times\PP^1$ are precisely
\[\left\{(-1, \pm 1), (0, \pm 1), (1, \pm 1), (3, \pm 5), (5, \pm 11), (56, \pm 419), (\infty, \infty)\right\}.\]
In particular, the paths defined over $\QQ$ are exactly those sequences with initial vertex $0$, $1$, or $-1$, and subsequent vertices $\pm 1$.  This shows that $\path(\QQ)$ is uncountable, but is contained in only three fibres of $\pi:\path\to\PP^1$, and consists entirely of finitely supported paths.
\end{example}
It is not difficult to extend this argument from integrality to show that, for many interesting correspondences over a number field $K$, the collection $\path(K)$ of $K$-rational paths is rather small, or at least rather pathological. Indeed, Example~\ref{ex:binary} can easily be generalized to a large class correpondences of the form $g(y)=f(x)$, with $g$ and $f$ polynomials, via results of Bilu and Tichy~\cite{bilutichy} and Avanzi and Zannier~\cite{zannier}. The aim of this remark is simply to convince the reader that one should really consider paths defined over $\overline{K}$, and that these are not \emph{a priori} defined over any finite extension of the base (see below).


\subsection{Specialization}

Our next result is a specialization theorem along the lines of Theorem~4.1 of \cite{callsilv}, but again the setting requires us to introduce some additional technical details. If $K$ is a number field and $U/K$ is a  curve, then we may consider a family of correspondences over $U$, that is a family $X\to U$ and $C\subseteq X\times_U X$ where, as usual, the projection maps are finite and surjective. Such an object is said to be \emph{polarized} if and only if there is an ample $L\in\Div(X)\otimes \RR$ and a real number $\alpha>1$ such that $y^*L-\alpha x^*L$ is linearly equivalent to a divisor fibral to the projection $C\to U$. One checks, then, that for all but finitely many $t\in U(\overline{K})$, the fibres $X_t, C_t, L_t$ give a polarized correspondence over $K(t)$.

Writing $t=\eta$ for the generic point of $U$, Theorem~4.1 of \cite{callsilv} gives for deterministic arithmetic dynamical systems the estimate
\[\hat{h}_{X, f, L}(P_t)=\left(\hat{h}_{X_\eta, f_\eta, L_\eta}(P_\eta)+o(1)\right)h_U(t)\]
where $h_U$ is any height on $U$ with respect to a divisor class of degree 1, and $o(1)\to 0$ as $h_U(t)\to \infty$ (see \cite{ghiocamav, var, henon, drinfeld} for improvements and variations in some special cases).

Just as one would not like to consider only paths defined over a number field, however, it turns out that considering only paths defined over $U$, or finite covers thereof, is overly restrictive; it is easy to construct schemes $\mathscr{U}\to U$ such that $\path(\mathscr{U})$ contains many interesting points which do not factor through finite $V\to U$. In simpler terms, elements of $\path(U)$ correspond to paths whose nodes are all $k(U)$-rational and, since function fields of varieties are generally not algebraically closed, this restriction rules out many interesting paths.

In order to give a wider class of paths to which our theorem will apply, we allow ourselves to consider $\mathscr{U}$-valued paths, for arbitrary $\mathscr{U}\to U$, as long as the initial vertex is a $U$-point.

\begin{theorem}
\label{th:specialization}
Fix a family of polarized correspondences over $U$, as above, a dominant morphism of schemes $\phi:\mathscr{U}\to U$, and a path $P\in \path(\mathscr{U})$ with $\pi(P)\in X(U)\subseteq X(\mathscr{U})$. Then for any height $h_U$ on $U$ relative to a divisor class of degree one,
\[\hat{h}_{X_{\phi(t)}, C_{\phi(t)}, L_{\phi(t)}}(P_t)=\left(\hat{h}_{X_{\phi(\eta)}, C_{\phi(\eta)}, L_{\phi(\eta)}}(P_\eta)+o(1)\right)h_{U}(\phi(t)),\]
where $o(1)\to 0$ as $h_{U}(\phi(t))\to \infty$. In particular, the height on the generic fibre may be computed as
\[\hat{h}_{X_{\phi(\eta)}, C_{\phi(\eta)}, L_{\phi(\eta)}}(P_\eta)=\lim_{h_{U}(\phi(t))\to\infty}\frac{\hat{h}_{X_{\phi(t)}, C_{\phi(t)}, L_{\phi(t)}}(P_t)}{h_{U}(\phi(t))}.\]
\end{theorem}

\begin{corollary}
With the assumptions of Theorem~\ref{th:specialization}, let
\[Z = \{t\in \mathscr{U}(\overline{K}):P_t\text{ is finitely supported}\}.\]
Then either $\phi(Z)\cap U(K)$ is finite, or else $\hat{h}_{X_{\phi(\eta)}, C_{\phi(\eta)}, L_{\phi(\eta)}}(P_\eta)=0$.
\end{corollary}


\subsection{Local canonical heights}
The construction of the canonical height above parallels that of Call and Silverman very closely, but in the local theory we see a more significant divergence. One of the fundamental facts about varieties $X$ over fields $K$ is that
\begin{equation}\label{eq:finitetype}X(\overline{K})=\bigcup_{[L:K]<\infty}X(L).\end{equation}
Indeed, this is true whenever $X$ is a $K$-scheme of finite type. 
This means that in the development of local heights for varieties, it is always possible to work over number fields; the theory over the algebraic closure follows once one shows that the values of the functions are independent of the particular extension.

Since $\path$ is not of finite type, \eqref{eq:finitetype} does not hold. Indeed, 
 Example~\ref{ex:binary} and its generalizations
proffer a bounty of instances in which, for every finite extension $L/K$, the paths in $\path(L)$ consist of fairly pathological examples, far from the generic behaviour in $\path(\overline{K})$. 
To remedy this, we employ Gubler's theory of $M$-fields. The following is essentially present in \cite{gubler}.

\begin{theorem}\label{th:measure}
Let $k$ be the algebraic closure of a number field, and let $M_k$ denote the usual set of places of $k$. Then there exists a $\sigma$-algebra $\mathscr{F}$ of subsets of $M_{k}$ and a $\sigma$-finite measure $\mu:\mathscr{F}\to\RR$ such that the following holds: For any finite $L/K$ and any place $v\in M_L$, the set
\[A_{v, L}=\{w\in M_{k}:v\text{ is the restriction to }L\text{ of }w\}\]
is measurable,
and
\begin{equation}\label{eq:localdegree}\mu(A_{v, L})=\frac{[L_v:K_v]}{[L:K]}.\end{equation}
\end{theorem}

The ratio in~\eqref{eq:localdegree} appears frequently in weighted sums in the theory of heights, and one should think of integrating with respect to this measure as a version of these weighted sums. In particular, for any non-zero $\alpha\in k$, the functions $v\mapsto \log|\alpha|_v$ and $v\mapsto\log^+|\alpha|_v$ are both measurable, and satisfy
\[\int_{M_k}\log|\alpha|_vd\mu(v)=0\]
(essentially the product formula) and
\[\int_{M_k}\log^+|\alpha|_vd\mu(v)=h(\alpha).\]
This construction allows us to present a theory of local heights for correspondences. 
\begin{theorem}
\label{th:localheights}
Let $C$ be a correspondence on $X$, polarized by an ample class $L$ and a real $\alpha>1$, let $\mathscr{W}=\path\setminus\pi^*(L)$, and fix a local height $\lambda_{X, L}$ on $X$ relative to $L$. Then there exists a function \[\hat{\lambda}_{X, C, L}:\mathscr{W}\times M_{\overline{K}}\to\RR\] such that the following hold:
\begin{enumerate}
\item $\hat{\lambda}_{X, C, L}(\cdot, v):\mathscr{W}(\overline{K})\to \RR$ is $v$-adically continuous.
\item \[\hat{\lambda}_{X, C, L}(P, v)=\lambda_{X, L}\left(\pi(P), v\right)+O(1),\]
where the implied constant depends on the place.
\item For some $f\in k(C)\otimes\RR$ satisfying $y^*L - \alpha x^*L=\operatorname{div}(f)$, we have
\[\hat{\lambda}_{X, C, L}(\sigma(P), v)=\alpha\hat{\lambda}_{X, C, L}(P, v)+\log|f\circ \epsilon(P)|_v,\]
where $\epsilon:\path\to C$ takes a path to the point on $C$ which gives the first edge. 
\item For $P\in\path(\overline{K})$ function $\hat{\lambda}_{X, C, L}(P, \cdot)$ is measurable, and
\[\hat{h}_{X, C, L}(P)=\int_{M_{\overline{K}}}\hat{\lambda}_{X, C, L}(P, v)d\mu(v).\]
\end{enumerate}
\end{theorem}

\subsection{Questions}
Before proceeding with a more formal development, we pose a few natural questions arising from the results in this paper. For this discussion, we fix a number field $K$, a projective variety $X$ (which, for the sake of simplicity, we will assume to be $\PP^N$), and a correspondence $C$ on $X$.

\begin{question}\label{q:canheightzero}
If the correspondence $C$ is polarized by $L$, for which $P\in\path(\overline{K})$ do we have $\hat{h}_{X, C, L}(P)=0$?
\end{question}

In the case of deterministic arithmetic dynamics, this is fairly well understood, and even in the present context there are some obvious sufficient conditions. But we also produce a class of examples in Section~\ref{sec:globcan} which contradict what might be an initial naive answer to Question~\ref{q:canheightzero}.

In order to pose the next question, we note that it follows from Corollary~\ref{cor:finitenessforrepetitive} that there exists an $B$ such that, if $P\in \path(\overline{K})$ is preperiodic, and $\pi(P)\in X(K)$, then there is a path $Q\in \path(\overline{K})$ with $\pi(P)=\pi(Q)$, and $\sigma^n(Q)=\sigma^{m}(Q)$ for some $0\leq m<n\leq B$. In other words, a preperiodic path whose initial vertex is $K$-rational must share this initial vertex with a preperiodic path of bounded length. 
\begin{question}\label{q:ubc}
How uniform can we make the value $B$ in the above remark? Is there a value which depends only on the degrees of the coordinate projections, and the degree of $K$ (in the case $X=\PP^N$)? Note that this is a generalization of the uniform boundedness conjecture of Morton and Silverman~\cite{mortsilv}.
\end{question}


\section{Geometry of correspondences}\label{sec:geom}

In this section we define some notions needed in the geometry of a correspondence. We may proceed with the following constructions in any category admitting fibre products and inverse limits, but in order to keep things tangible we work in the category of $S$-schemes, for some fixed  scheme $S$. It is tempting to work in the category of varieties over a field, but this too restrictive to allow a satisfactory construction of the path space. The reader who is not interested in the content of Section~\ref{sec:var} could take $S$ to be the spectrum of a field.

Let $X$ be a separated, integral $S$-scheme of finite type (i.e., a variety if $S$ is the spectrum of a field), and let $C\subseteq X\times_S X$ be a subscheme such that the two projection maps are finite and surjective. Our first goal is to construct the path space $\path_C$ (we omit the subscript when the context makes it obvious) of the correspondence on $X$ defined by these data. To do this, we construct a sequence of $S$-schemes parametrizing paths of finite length, and take $\path$ to be their inverse limit. There is then a natural interpretation of the correspondence as a dynamical system on $\path$. In particular, there is a \emph{shift map} $\sigma:\path\to\path$ which corresponds to deleting the first vertex in a path.

\begin{theorem}
\label{prop:paths}
Let $C$ be a correspondence on the $S$-scheme $X$. Then there exists a separated, integral $S$-scheme $\path$, surjective morphisms $\pi:\path\to X$, $\epsilon:\path\to C$, and a finite surjective morphism $\sigma:\path\to \path$ making the following diagram commute.
\[\xymatrix{
\path \ar[d]_{\pi} \ar[rr]^{\sigma} \ar[dr]^{\epsilon} &  &  \path \ar[d]^\pi\\  
X & C \ar[l]^x \ar[r]_y & X
}\]
\end{theorem}

For a given path $P\in\path$, we should think of $\pi(P)\in X$ as the initial vertex of $P$, $\epsilon(P)\in C$ as the initial edge, and $\sigma(P)\in\path$ as the path obtained by forgetting both of these (or shifting along the path). Thus, the commutative diagram in the statement of the theorem says roughly that the terminal vertex of the initial edge of a path is the initial vertex of the path shifted once (and that the initial vertex of the initial edge is also the initial vertex of the entire path).

\begin{proof}[Proof of Theorem~\ref{prop:paths}]
We define a sequence of $S$-schemes $\path_n$, along with morphisms $\pi_n, \tau_n:\path_n\to X$, as follows. First, we take $\path_1=X$, with $\pi_1=\tau_1$ the identity map. Now, for each $n$, we define $\path_{n+1}$ to be the base extension
\[\xymatrix{
\path_{n+1} \ar[r]^{\beta_n} \ar[d]_{\alpha_n} &  C \ar[d]^x\\ 
\path_n \ar[r]_{\tau_n}  &  X 
}\]
The maps $\pi_{n+1}, \tau_{n+1}$ are defined by  $\tau_{n+1}=y\circ \beta_n$, and $\pi_{n+1}=\pi_n\circ\alpha_n$.

Intuitively, $\path_n$ parametrizes paths of length $n$, with initial and terminal vertices given by $\pi_n$ and $\tau_n$, respectively. By the fibre product construction, $\path_{n+1}$ parametrizes pairs consisting of a path of length $n$ and an edge, such that the terminal vertex of the path matches the initial vertex of the edge; in other words, a path of length $n+1$.
Note that, by the standard properties of fibre products, each $\path_n$ is a separated, integral scheme of finite type over $S$. 

Now consider the inverse system of $S$-schemes
\[\xymatrix{
\cdots \ar[r]^{\alpha_{n+1}} & \path_{n+1} \ar[r]^{\alpha_n} & \cdots   \ar[r]^{\alpha_{2}}  & \path_2 \ar[r]^{\alpha_1} &\path_1,
}\]
and let $\path=\varprojlim \path_n$ (see \cite[Chapter~31]{stacks-project}). Note that $\path$ is not, in general, of finite type over $S$.  That $\path$ is a separated, integral $S$-scheme follows simply from the fact that these properties are preserved under base extension and taking limits.

The inverse limit construction gives morphisms $\path\to\path_n$ for every $n$, which are surjective and commute with the transition maps $\alpha_n$. Since $\path_1\cong X$ and $\path_2\cong C$, we thus have surjective maps $\pi:\path\to X$ and $\epsilon:\path\to C$ which satisfy $\pi=x\circ\epsilon$ (the left-hand-side of the claimed commutative diagram).

It remains to construct $\sigma:\path\to\path$. Note that, by diagram chasing, we have $C\times_X\path_n\cong \path_{n+1}$, with respect to the maps $y:C\to X$ and $\pi_n:\path_n\to X$. In other words, paths of length $n+1$ might just as well be constructed from paths of length $n$ by adding an edge at the beginning. Note that the resulting map $\path_{n+1}\to C$ corresponds to $\alpha_2\circ\cdots\circ\alpha_n$ under the association $C\cong \path_2$. Taking limits, this gives (by \cite[Chapter~39, Lemma~2.3]{stacks-project})
\[C\times_X \path = C\times_X\varprojlim \path_n = \varprojlim \left(C\times_X \path_n\right)\cong \varprojlim\path_{n+1}=\path,\]
where the latter is viewed as an $X$-scheme by $\pi:\path\to X$, and $C$ is viewed as an $X$-scheme by $y:C\to X$. Furthermore, the implied map $C\times_X \path \to C$ coincides, from the construction, with $\epsilon:\path\to C$.  In other words, we have a base extension diagram 
\[\xymatrix{
\path \ar[d]_{\epsilon} \ar[r] &\path \ar[d]^{\pi} \\
C \ar[r]_y & X,
}\]
and we call the top morphism $\sigma$. The fact that $\sigma:\path\to\path$ is finite, separated, and surjective follows from it being a base extension of a finite, surjective morphism of varieties.

\end{proof}

The proof of the finiteness of $\sigma:\path\to\path$ belies the simplicity of the motivation. Essentially, given a path $Q$ there should exist only finitely many paths $P$ with $\sigma(P)=Q$ simply because there are only finitely many points on $X$ which could be the initial vertex of a path whose second vertex is $\pi(Q)$. This follows just because the two projection maps $C\to X$ are finite.

In deterministic dynamics, the notion of preperiodicity is very important. There are numerous possible generalizations of this to the dynamics of correspondences, and we name three.
\begin{definition}
Let $X$ be a projective variety, let $C$ be a correspondence on $X$, let $\path$ be the associated path space, and let $P\in\path$. We say that $P$ is
\begin{enumerate}
\item \emph{preperiodic} if and only if there exist $n\neq m$ with $\sigma^n(P)=\sigma^m(P)$;
\item \emph{finitely supported} if and only if the quantity $\pi\circ\sigma^n(P)$ takes only finitely many values as $n$ varies; and
\item \emph{repetitive} if and only if there exist $n\neq m$ with $\pi\circ\sigma^n(P)=\pi\circ\sigma^m(P)$.
\end{enumerate}
We say that $a\in X$ is \emph{constrained} relative to the correspondence $C$ if and only if there exists a repetitive path $P\in\pi^{-1}(a)$.
\end{definition}

\begin{remark}
It is easy to see that every preperiodic path is finitely supported, and every finitely supported path is repetitive, but in general the three conditions are not equivalent (as they are in the deterministic case). Example~\ref{ex:binary} contains paths which are finitely supported but not preperiodic, while Example~\ref{ex:heights} below gives a path which is repetitive but not finitely supported. Note, however, that the condition that there exists a repetitive path $P$ with $\pi(P)=a$ \emph{is} equivalent to the condition that there exists a preperiodic (and hence finitely supported) path $P$ with $\pi(P)=a$.
\end{remark}

\begin{remark}
For each $n\geq 2$, let $\path_{n}$ be the $S$-scheme parametrizing paths of length $n$ as above, with initial and terminal vertices given by the maps $\pi_n, \tau_n:\path_n\to X$. This gives a morphism $(\pi_n, \tau_n):\path_n\to X\times_S X$. The points of $X$ which land in the image of this morphism under the diagonal embedding $X\to X\times_S X$ are precisely the initial vertices of cyclic paths of length $n-1$. Since we should expect $\path_n$ to have the same dimension as $X$, one should typically expect this intersection to be 0-dimensional, but additional hypotheses are needed to make this concrete.
\end{remark}


\section{The global canonical height}\label{sec:globcan}

Let $X$ be a projective variety over $K$, and let $C$ be a correspondence over $X$. Recall that $C$ is \emph{polarized} if and only if there is an ample (Cartier) divisor $L\in\Div(X)\otimes\RR$ and a real number $\alpha>1$ such that $y^*L\sim \alpha x^*L$, where $x, y:C\to X$ are the coordinate projections. In this section we will prove Theorem~\ref{th:canheight}.

Before proceeding with the proof, we note that nowhere do we use the condition that $L$ is ample. This condition arises in Corollary~\ref{cor:finitenessforrepetitive}, and is likely to play a role in the determination of where the canonical height vanishes.

\begin{proof}[Proof of Theorem~\ref{th:canheight}]
Let $P\in \path(\overline{K})$.
By the standard properties of Weil heights~\cite[Theorem~B.3.2, p.~184]{hind-silv}, and the commutative diagram in Proposition~\ref{prop:paths}, we have
\begin{eqnarray*}
h_{X, L}\left(\pi\circ\sigma(P)\right)&=&h_{X, L}(y\circ\epsilon(P))\\
&=&h_{C, y^*L}(\epsilon(P))+O(1)\\
&=&\alpha h_{C, x^*L}(\epsilon(P))+O(1)\\
&=&\alpha h_{X, L}(x\circ \epsilon(P))+O(1)\\
&=&\alpha h_{X, L}(\pi(P))+O(1).
\end{eqnarray*}
(Note in particular, that although $P$ is not a point on a projective variety, both $\pi(P)$ and $\epsilon(P)$ are, and so we are not applying the height machine out of context.) Thus it follows that
\[\left|\alpha^{-1}h_{X, L}(\pi\circ\sigma(P))-h_{X, L}(\pi(P))\right|\leq \kappa\]
for some constant $\kappa\geq 0$, for all $P\in\path(\overline{K})$. We now apply a standard telescoping sum argument, due to Tate (seen also in \cite{callsilv}). By the triangle inequality, we have
\begin{equation}\label{eq:tatebound}\left|\alpha^{-n}h_{X, L}(\pi\circ\sigma^n(P))-h_{X, L}(\pi(P))\right|\leq \left(1+\alpha^{-1}+\cdots+\alpha^{-n+1}\right)\kappa\leq \left(\frac{\alpha}{\alpha-1}\right)\kappa.\end{equation}
Applying this with $\sigma^m(P)$ in place of $P$, we have
\[\left|\alpha^{-(m+n)}h_{X, L}(\pi\circ\sigma^{m+n}(P))-\alpha^{-m}h_{X, L}(\pi\circ\sigma^m(P))\right|\leq \alpha^{-m} \left(\frac{\alpha}{\alpha-1}\right)\kappa.\]
This shows that the sequence $\alpha^{-n}h_{X, L}(\pi\circ\sigma^n(P))$ of real numbers is Cauchy,
whereupon the limit
\[\hat{h}_{X, C, L}(P)=\lim_{n\to\infty}\alpha^{-n}h_{X, L}(\pi\circ\sigma^n(P))\]
exists.  The first property in the theorem follows immediately from the definition, while the second follows by taking the limit of \eqref{eq:tatebound} as $n\to\infty$.
\end{proof}

It is easy to see from the definition that $\hat{h}_{X, C, L}(P)=0$ when $P$ is finitely supported (and so when $P$ is preperiodic), although this is not generally the case when $P$ is merely repetitive, as the following example shows.
\begin{example}\label{ex:heights} Let $C\subseteq \PP^1\times \PP^1$ be the correspondence defined by (the closure of) the curve $y^2=x^3+1$. This correspondence is polarized, with $L$ any ample divisor on $\PP^1$ and $\alpha=3/2$.

 It is not hard to check, by making explicit the bound on the difference between the canonical height and the Weil height, that any path $P$ which begins as
\[P:0\to 1\to 2^{1/2}\to (2^{3/2}+1)^{1/2}\to\cdots\]
satisfies $\hat{h}_{X, C, L}(P)>0$. But now note that the repetitive path
\[Q:0\to -1 \to 0 \to 1\to 2^{1/2}\to \cdots \]
satisfying $\sigma^2(Q)=P$ must have $\hat{h}_{X, C, L}(Q)=\frac{4}{9}\hat{h}_{X, C, L}(P)>0$. Indeed, we may use this example  to produce paths $Q\in\path(\overline{\QQ})$ of arbitrarily small positive canonical height with $\pi(Q)\in\PP^1(\QQ)$ (indeed, with $\pi(Q)=0$). This is a phenomenon not present in deterministic arithmetic dynamics.
\end{example}

\begin{remark}
\`{A} propos of the discussion that many interesting $\overline{K}$-points of $\path$ are not defined over any number field, it is worth noting that a path $P\in \path(\overline{K})$ is finitely supported if and only if we have both $\hat{h}_{X, C, L}(P)=0$ and $P\in \path(E)$ for some finite extension $E/K$. The question of when the vanishing of the canonical height implies finite support of a path is thus intimately linked to the question of which paths are defined over some number field. Note that if $\hat{h}_{X, C, L}(P)=0$ but $P$ is not $E$-rational for any finite $E/K$, then the points $\pi\circ\sigma^n(P)$ have unbounded algebraic degree, and height zero with respect to a certain Weil height on $X$; these are precisely the antecedent conditions of many equidistribution results.
\end{remark}

In general, it is not clear what conditions are necessary for $\hat{h}_{X, C, L}(P)=0$, even over number fields (as opposed to functions fields, in which context we expect the answer to be even more subtle \cite{baker, benedetto}). For instance, let $X=\PP^1$ and let $C\subseteq \PP^1\times\PP^1$ be the correspondence defined by the closure of $y^2=x^3$. Then any $P\in\path(\overline{\QQ})$ with $\pi(P)$ a root of unity satisfies $\hat{h}_{X, C, L}(P)=0$, but there are certainly non-repeating paths of this form.

More generally, if $f, g:X\to X$ are two finite morphisms whose set of preperiodic points coincide, then the correspondence defined on $X$ by $g(y)=f(x)$ admits non-repeating paths of canonical height zero. In light of characterizations such as those in \cite[Theorem~1.2]{bakerdemarco}, it would be of interest to determine whether or not these are the only sorts of examples.

Before concluding this section, we note that there is a natural topology on $\path(\overline{K})$, which we will call the \emph{tree topology}. In particular, equip $\path_n(\overline{K})$ with the discrete topology for all $n$, and define the tree topology on $\path_n(\overline{K})$ to be the weakest which makes each of the projection maps $\path(\overline{K})\to\path_n(\overline{K})$ continuous. A base for this topology is simply the collection of sets
\[U_{P, n}=\left\{Q\in \path(\overline{K}):\pi\circ\sigma^m(Q)=\pi\circ\sigma^m(P)\text{ for all }0\leq m\leq n\right\},\]
for $n\geq 0$ and $P\in\path(\overline{K})$. It is natural to ask how the canonical height behaves with respect to this topology.
\begin{proposition}
The function $\hat{h}_{X, C, L}:\path(\overline{K})\to \RR$ is continuous with respect to the tree topology, and each fibre $\path_a(\overline{K})=\pi^{-1}(a)\subseteq \path(\overline{K})$ is  compact. 
\end{proposition}

\begin{proof}
The compactness of $\path_a(\overline{K})$ can be obtained as a general fact about inverse limits of compact spaces, since $\path_a(\overline{K})$ is the inverse limit of $\path_{a, n}(\overline{K})$, a finite discrete space for each $n$.

Note that we have constructed, above, a constant $\kappa'\geq 0$ such that for any $P\in \path(\overline{K})$, we have
\[\left|\hat{h}_{X, C, L}(P)-h_{X, L}(\pi(P))\right|\leq \kappa'.\]
Now, supposing that $Q\in U_{P, n}$, we have
\begin{eqnarray*}
\left|\hat{h}_{X, C, L}(P)-\hat{h}_{X, C, L}(Q)\right|&=&\alpha^{-n}\left|\hat{h}_{X, C, L}(\sigma^n(P))-\hat{h}_{X, C, L}(\sigma^n(Q))\right|\\
&\leq &\alpha^{-n}\left|h_{X, L}(\pi\circ\sigma^n(P))-h_{X, L}(\pi\circ\sigma^n(Q))\right|\\
&&+\alpha^{-n}\left|\hat{h}_{X, C, L}(\sigma^n(P))-h_{X, L}(\pi\circ\sigma^n(P))\right|\\
&&+\alpha^{-n}\left|\hat{h}_{X, C, L}(\sigma^n(Q))-h_{X, L}(\pi\circ\sigma^n(Q))\right| \\
&\leq& 2\alpha^{-n}\kappa',
\end{eqnarray*}
since $\pi\circ\sigma^n(Q)=\pi\circ\sigma^n(P)$. In other words, given any $\epsilon>0$ there exists an $n$ such that
\[\left|\hat{h}_{X, C, L}(P)-\hat{h}_{X, C, L}(Q)\right|<\epsilon\]
for all $Q\in U_{P, n}$, whence the claim.
\end{proof}

\begin{remark}
In single-valued arithmetic dynamics \cite{callsilv}, the canonical height is defined on the base variety $X$, where here we have defined it on the path space $\path$.
In light of the above proposition, it is natural to define two real-valued functions on $X(\overline{K})$ itself by
\begin{gather*}
\hat{h}_{X, C, L, \mathrm{max}}(a)=\max\left\{\hat{h}_{X, C, L}(P):\pi(P)=a\right\}\\
\hat{h}_{X, C, L, \mathrm{min}}(a)=\min\left\{\hat{h}_{X, C, L}(P):\pi(P)=a\right\}.\end{gather*}
It is clear from Theorem~\ref{th:canheight} that both of these differ from any Weil height (on $X$ relative to $L$) by a bounded amount, and many interesting questions about the canonical height can phrased in terms of $\hat{h}_{X, C, L, \mathrm{min}}$ in particular.
\end{remark}


\section{Specialization}\label{sec:var}

We now state the specialization theorem, for which we must set the stage.
Let $K$ be a number field, let $U/K$ be a curve, and let $k=\overline{K}$. By a \emph{family of polarized correspondences} over $U$, we mean simply a variety $X\to U$ with a correspondence $C\subseteq X\times_U X$, a divisor $L\in \Div(X)\otimes \RR$, and a real number $\alpha>1$ such that $y^* L-\alpha x^*L$ is linearly equivalent to a divisor fibral to the projection $C\to U$. This is equivalent to requiring the generic fibre in the family to be a polarized correspondence.   From this definition it follows easily that there is an affine open subset $U_0\subseteq U$ such that for all $t\in U_0$ the triple of fibres $(X_t, C_t, L_t)$ defines a polarized correspondence as defined above. We will always take $U_0=U$, paring down to an open subset when necessary, although when we speak of heights on $U$ we of course mean heights on a projective curve birational to $U$.

From Section~\ref{sec:geom}, we can construct the path space $\path$ of the correspondence as a $U$-scheme. Just as in the arithmetic case, however, one should expect $\path(U)$ to be rather small and uninteresting, and similarly for $\path(V)$ for any $V\to U$ of finite type. 

 In analogy to the arithmetic case, let $\phi:\mathscr{U}\to U$ be a scheme not necessarily of finite type over $U$, with $\phi$ dominant, and consider paths $P\in \path(\mathscr{U})$ such that $\pi(P)\in X(U)$ (in other words, such that the map $\pi(P):\mathscr{U}\to X$ factors through $U$). We can specialize such a path at points $t\in\mathscr{U}(k)$, and ask how $\hat{h}_{X_{\phi(t)}, C_{\phi(t)}, L_{\phi(t)}}(P_t)$ varies as a function of $t$.  We should note that the relative non-specificity of $\mathscr{U}$ is misleading. Let $\mathscr{U}'\to U$ be the limit of the inverse system of curves $V$ with a finite map $V\to U$. Then it is straightforward to show that any $P:\mathscr{U}\to\path$ such that $\pi(P):\mathscr{U}\to X$ factors through $U$, must in turn factor through some $P':\mathscr{U}'\to \path$. Without loss of generality, then, we will always suppose that $\mathscr{U}=\mathscr{U'}$. This is the geometric equivalent of working over $\overline{K}$, rather than an arbitrary field extension of $K$.

Recall that the generic fibre (over $U$) of a family of correspondences as above is a polarized correspondence defined over the function field $k(U)$, which we denote by $X_\eta, C_\eta, L_\eta$. There is a natural height on $X_\eta(\overline{k(U)})$ relative to $L_\eta$, corresponding on $k(U)$ to the degree of a function on $U$, and it is easy to check that the construction of the canonical height in Section~\ref{sec:globcan}  works in this setting as well. In particular, given any finite cover $V\to U$ and any morphism $f:V\to X$, we have a corresponding point $f_\eta\in X_\eta(\overline{k(U)})$ given by some embedding $k(V)\to \overline{k(U)}$, and we define $h_{X_\eta, L_\eta}(f_\eta)$ accordingly; although the notation $f_\eta$ is ambiguous, depending on the choice of embedding, the height is well-defined. We denote by $\hat{h}_{X_{\phi(\eta)}, C_{\phi(\eta)}, L_{\phi(\eta)}}(P_\eta)$ the canonical height as computed on the generic fibre (noting that, since $\phi$ is dominant and $\eta$ is the generic point of $\mathscr{U}$, $\phi(\eta)$ is the generic point of $U$).

Theorem~\ref{th:specialization}, generalizing a result of the same name due to Call and Silverman~\cite{callsilv}, establishes that the canonical height varies nicely in such a family.
\begin{theorem}\label{th:specialization}
In the setting above, let $h_U$ be a Weil height on $U$ corresponding to any divisor class of degree $1$. Then
\[\hat{h}_{X_{\phi(t)}, C_{\phi(t)}, L_{\phi(t)}}(P_t)=\left(\hat{h}_{X_{\phi(\eta)}, C_{\phi(\eta)}, L_{\phi(\eta)}}(P_\eta)+o(1)\right)h_U(\phi(t)),\]
where $o(1)\to 0$ as $h_U(\phi(t))\to \infty$. In particular, if $\hat{h}_{X_{\phi(\eta)}, C_{\phi(\eta)}, L_{\phi(\eta)}}(P_\eta)>0$, then the set
\[\{\phi(t):P_t\text{ is finitely supported}\}\subseteq U(k)\]
is a set of bounded height.
\end{theorem}

The proof of the theorem comes in a series of lemmas, and follows the outline of the proof of Theorem~4.1 of \cite{callsilv}. 
The following lemma is the analogue of \cite[Theorem~3.1]{callsilv}, and is proven in the same way.
\begin{lemma}
Let $X, C, L$ be a family of polarized correspondences over $U$, and let $h_U$ be a degree-one height on $U$. Then there exist constants $c_1$ and $c_2$ such that for all $t\in U(\overline{K})$ and all $P\in\path_{t}(\overline{K})$,
\[\left|\hat{h}_{X_t, C_t, L_t}(P)-h_{X_t, L_t}(\pi(P))\right|\leq c_1h_U(t)+c_2\]
\end{lemma}

\begin{proof}
Write $\tilde{U}$ for the projective closure of $U$. 
Note that this lemma is the analogue of \cite[Theorem~3.1]{callsilv}, but the proof is in fact simpler. We need not worry about resolving the indeterminacy of rational maps, since the coordinate projections are always morphisms. Now, by the hypothesis that $y^*L -\alpha x^* L$ is fibral, as in \cite[Proof of Theorem~3.1]{callsilv} there is an effective divisor $D$ on $\tilde{U}$ such that for the map $u:C\to U$ we have
\[u^* D > y^* L-\alpha x^*L > -u^* D.\]
Choosing $H$ on $\tilde{U}$ ample such that $nH+D$ and $nH-D$ are both ample, we have
\[\left|h_{C, y^*L -\alpha x^*L}(Q)\right|\leq nh_{\tilde{U}, H}(u(Q))+O(1)\]
for all $Q$ on $Y$. Thus, for a given path $P\in\path(U)$,
\begin{eqnarray*}
\left|h_{X_t, L_t}(\pi\circ\sigma(P_t))-\alpha h_{X_t, L_t}(\pi(P_t))\right|&=&\left|h_{X, L}(y\circ\epsilon(P_t))-\alpha h_{X_t, L_t}(x\circ\epsilon(P_t))\right|\\
&=&\left|h_{Y, y^*L}(\epsilon(P_t))-\alpha h_{Y, x^*L}(\epsilon(P_t))\right|\\
&=&\left|h_{C, y^*L - \alpha x^*L}(\epsilon(P_t))\right|\\
&\leq& nh_{\tilde{U}, H}(u\circ\epsilon(P_t))+O(1)\\
&=&nh_{\tilde{U}, H}(t)+O(1).
\end{eqnarray*}
The result now follows from the fact that $h_{\tilde{U}, H}$ is bounded above by a multiple of any given degree one height on $U$.
\end{proof}

\begin{lemma}
Let $\psi:V\to U$ be a finite cover, and let $f:V\to X$ be a morphism. Then for any $t\in V$,
\[h_{X_{\psi(t)}, L_{\psi(t)}}(f(t))=h_{X_\eta, L_\eta}(f_\eta)h_U(\psi(t))+O\left(\sqrt{h_{U}(\psi(t))}\right),\]
where the implied constant is allowed to depend on $f$.
\end{lemma}

\begin{proof}
Essentially by definition,
\[h_{X_\eta, L_\eta}(f_\eta)=\frac{\deg(f^*L)}{[k(V):k(U)]}=\frac{\deg(f^*L)}{\deg (\psi)}.\]
Let $D$ be the divisor relative to which $h$ is defined, and let
\[E=\frac{1}{h_{X_\eta, L_\eta}(f_\eta)}f^*L,\]
so that $\deg(E)=\deg(D)=1$.
 Note that, by \cite[Proposition 5.4, p.~115]{lang}, we have
\[h_{U, E}(t)=h_{U, D}(t)+O\left(\sqrt{h_{U, D}(t)}\right).\]
 Also, the embedding $X_t\to X$ pulls $L$ back to $L_t$, and so for any $t\in V$ we have, 
\begin{eqnarray*}
h_{X_t, L_t}(f(t))&=&h_{X, L}(f(t))\\&=&h_{V, f^*L}(t)+O(1)\\
&=&h_{X_\eta, L_\eta}(f_\eta)h_{U, E}(t)+O(1)\\
&=&h_{X_\eta, L_\eta}(f_\eta)h_{U, D}(\psi(t))+O\left(\sqrt{h_{U, D}(\psi(t))}\right).
\end{eqnarray*}
\end{proof}

\begin{proof}[Proof of Theorem~\ref{th:specialization}]
Let $P\in\path(\mathscr{U})$ such that $\pi(P)\in X(V)$, for some finite $\psi:V\to U$  that $\phi:\mathscr{U}\to U$ factors as $\phi=\psi\circ\xi$, for some $\xi:\mathscr{U}\to V$. We may specialize $P$ at $t\in \mathscr{U}(\overline{K})$ to obtain a path $P_t\in\path_{\phi(t)}(\overline{K})$. Note that, in this setting, $\pi(P_t)=\pi(P)_{\xi(t)}\in X(V)$. Then interpreting $\pi(P)$ as a morphism from $V$ to $X$, we have the following.
\begin{multline*}
\left|\hat{h}_{X_{\phi(t)}, C_{\phi(t)}, L_{\phi(t)}}(P_t)-\hat{h}_{X_{\phi(\eta)}, C_{\phi(\eta)}, L_{\phi(\eta)}}(P_\eta)h_U(\phi(t))\right|\\ \leq \left|\hat{h}_{X_{\phi(t)}, C_{\phi(t)}, L_{\phi(t)}}(P_t)-h_{X_{\phi(t)}, L_{\phi(t)}}(\pi(P_t))\right| \\
+ \left|h_{X_{\phi(t)}, L_{\phi(t)}}(\pi(P)_{\xi(t)})-h_{X_{\phi(\eta)}, L_{\phi(\eta)}}(\pi(P)_\eta)h_U(\psi\circ\xi(t))\right|\\+\left|h_{X_{\phi(\eta)}, L_{\phi(\eta)}}(\pi(P)_\eta)-\hat{h}_{X_{\phi(\eta)}, C_{\phi(\eta)}, L_{\phi(\eta)}}(P_\eta)\right|h_{U}(\phi(t))\\
\leq  c_3h_U(\phi(t))+c_4\sqrt{h_U(\phi(t))}+c_5,
\end{multline*}
where $c_4$ and $c_5$ depend on $V$ and $P$, but $c_3$ depends only on $X$, $C$, $L$, and $U$.
We thus have
\begin{equation}\label{eq:seethree}\limsup_{h_U(\phi(t))\to \infty}\left|\frac{\hat{h}_{X_{\phi(t)}, C_{\phi(t)}, L_{\phi(t)}}(P_t)}{h_U(\phi(t))}-\hat{h}_{X_{\phi(\eta)}, C_{\phi(\eta)}, L_{\phi(\eta)}}(P_\eta)\right|\leq c_3,\end{equation}
noting that this bound does not depend on $P$ or the cover $V\to U$.

Now, given our path $P\in \path(\mathscr{U})$ satisfying $\pi(P)\in X(U)$, note that $Q=\sigma^n(P)$ is a path $Q\in \path(\mathscr{U})$ satisfying $\pi(Q)\in X(V)$ for some finite $\psi:V\to U$ as above.
Thus we may apply the estimate~\eqref{eq:seethree} to the path $Q=\sigma^n(P)$, to obtain
\begin{multline*}\limsup_{h(\phi(t))\to\infty}\left|\frac{\hat{h}_{X_{\phi(t)}, C_{\phi(t)}, L_{\phi(t)}}(P_t)}{h_U(\phi(t))}-\hat{h}_{X_{\phi(\eta)}, C_{\phi(\eta)}, L_{\phi(\eta)}}(P_\eta)\right|\\=\limsup_{h_U(\phi(t))\to\infty}\alpha^{-n}\left|\frac{\hat{h}_{X_{\phi(t)}, C_{\phi(t)}, L_{\phi(t)}}(Q_t)}{h_U(\phi(t))}-\hat{h}_{X_{\phi(\eta)}, C_{\phi(\eta)}, L_{\phi(\eta)}}(Q_\eta)\right|\leq \alpha^{-n}c_3.\end{multline*}
Since this holds for all $n$, and since the absolute value is non-negative, we have in fact
\[\lim_{h_U(\phi(t))\to\infty}\left|\frac{\hat{h}_{X_{\phi(t)}, C_{\phi(t)}, L_{\phi(t)}}(P_t)}{h_U(\phi(t))}-\hat{h}_{X_{\phi(\eta)}, C_{\phi(\eta)}, L_{\phi(\eta)}}(P_\eta)\right|=0.\]
This proves the main claim.

If the set \[\{\phi(t):P_t\text{ is finitely supported}\}\subseteq U(k)\] is not of bounded height, then we may choose a sequence $t_n\in \mathscr{U}(k)$  such that $P_{t_n}$ is finitely supported for all $n$, and $h_U(\phi(t_n))\to \infty$ with $n$. But then we have
\[0=\lim_{n\to\infty}\left|\frac{\hat{h}_{X_{\phi(t)}, C_{\phi(t)}, L_{\phi(t)}}(P_t)}{h_U(\phi(t))}-\hat{h}_{X_{\phi(\eta)}, C_{\phi(\eta)}, L_{\phi(\eta)}}(P_\eta)\right|=\hat{h}_{X_{\phi(\eta)}, C_{\phi(\eta)}, L_{\phi(\eta)}}(P_\eta)\]
since the fractional term always has vanishing numerator.
\end{proof}

As a corollary to the theorem, we note the following.
\begin{corollary}
In the setting of Theorem~\ref{th:specialization}, consider a dominant morphism $\phi:\mathscr{U}\to U$ and a path $P\in \path(\mathscr{U})$ with $\hat{h}_{X_{\phi(\eta)}, C_{\phi(\eta)}, L_{\phi(\eta)}}(P_\eta)>0$. Then the set
\[\{t\in U(K):P_s\text{ is finitely supported for some }\phi(s)=t\}\]
is a finite set.
\end{corollary}
It would be of significant interest to determine the conditions under which $\hat{h}_{X_{\phi(\eta)}, C_{\phi(\eta)}, L_{\phi(\eta)}}(P_\eta)=0$, but experience indicates that this is likely more challenging than the yet-unresolved number field case \cite{baker, benedetto}.


\section{Measured sets of absolute values and augmented divisors}

In this section we lay out some of the basics needed for our discussion of local canonical heights. The fundamental difference between this treatment and that in \cite{hind-silv, lang} is that we work in a fundamental way over the algebraic closure $k=\overline{K}$ of a number field, which bars us from expressing global heights as weighted sums of local heights in the usual way. The remedy is to employ Gubler's theory of $M$-fields \cite{gubler} although, since the author developed this material independently before learning of Gubler's work, the presentation here differs slightly from \cite{gubler}.

Let $k$ be a field. A \emph{measured set of absolute values} on $k$ will be a set $M_k$ of absolute values, along with a $\sigma$-algebra $\mathcal{B}$ of subsets, and a measure $\mu$, such that for every non-zero $\alpha\in k$, the function $v\mapsto \log|\alpha|_v$ is measurable. We will say that such an object \emph{satisfies the product formula} if and only if we have
\[\int_{M_k}\log|\alpha|_vd\mu(v)=0\]
for any $\alpha\in k^*$.

Here is a simple example to aid the reader's intuition.
\begin{example}
Let $k=\QQ$, let $M_\QQ$ be the usual set of absolute values, let $\mathcal{B}$ be the set of all subsets of $M_\QQ$, and let $\mu$ be the counting measure. Then  $v\mapsto \log|\alpha|_v$ is measurable for any $\alpha\in\QQ^*$, since the function vanishes off of a finite set, and
\[\int_{M_k}\log|\alpha|_vd\mu(v)=\sum_{v\in M_\QQ}\log|\alpha|_v=0\]
by the standard product formula.

More generally, if $k$ is a number field, $M_k$ is the usual set of absolute values, $\mathcal{B}$ is the power set of $M_k$, and $\mu$ is a weighted counting measure assigning weight $[k_v:\QQ_v]/[k:\QQ]$ to each $v\in M_k$, then this is a measured set of absolute values on $k$ satisfying the product formula.
\end{example}

We begin with a theorem ensuring that, for any number field $K$, there is a natural way to view the set $M_k$ of places of $k=\overline{K}$ as a measured set of absolute values.
In order to make the result somewhat more generally applicable, however, we relax the condition that $K$ is a number field. We say that the set of absolute values $M_K$ on the field $K$ is \emph{coherent} if and only if, for every finite extension $E/L/K$ and every $v\in M_L$, we have
\begin{equation}\label{eq:coherent}\sum_{w\mid v}[E_w:L_w]=[E:L],\end{equation}
where the sum is over extensions $w$ of $v$ to $E$. For instance, the usual set of absolute values on a number field is coherent, as is the usual set of absolute values on a function field in characteristic zero.
\begin{theorem}\label{th:measuredplaces}
Let $K$ be a field, let $M_K$ be a coherent set of absolute values on $K$, and let  $M_{k}$ be the set of absolute values on $k=\overline{K}$ which restrict on $K$ to elements of $M_K$. Then we may equip $M_{k}$ with the structure of a measured set of absolute values in such a way that for any finite $L/K$ and $w\in M_{k}$, the set
\[A_{w, L}=\{v\in M_{k}:v=w\text{ on }L\}\]
is measurable, and
\[\mu(A_{w, L})=\frac{[L_w:K_w]}{[L:K]}.\]
\end{theorem}

\begin{proof}
Fix an absolute value $v\in M_K$, and let $M_{k, v}$ be the set of absolute values on $k$ restricting on $K$ to $v$. We write $\mathscr{F}_{0, v}$ for the algebra of sets generated by sets of the form $A_{w, L}$, for $L/K$ finite and $w\in M_L$, and define $\mu_v$ on $\mathscr{F}_{0, v}$ as in the statement of the theorem, extended to disjoint unions by additivity. Note that the fact that $\mu_v$ is well-defined is precisely the coherence condition imposed on the places of $K$. In particular, if $E/L/K$ is a tower of finite extensions, then
\[\mu_v(A_{w, L})=\frac{[L_w:K_w]}{[L:K]}=\sum_{\substack{u\in M_E\\ u=w\text{ on }L}}\frac{[E_u:K_u]}{[E:K]}=\sum_{\substack{u\in M_E\\ u=w\text{ on }L}}\mu_v(A_{u, L}).\]
We wish to show that $\mu_v$ extends to a measure on some $\sigma$-algebra, and so we recall a special case of a theorem of Carath\'{e}odory.
\begin{theorem}[Carath\'{e}odory Extension Theorem {\cite[p.~97]{thomson}}]\label{th:cara}
Let $X$ be a set, let $\mathcal{B}$ be an algebra of subsets of $X$, and let $\mu:\mathcal{B}\to \RR$ be a function such that $\mu(\emptyset)=0$,  and
\begin{equation}\label{eq:sigmaadd}\mu\left(\bigcup_{i\geq 0}B_i\right)=\sum_{i\geq 0}\mu(B_i)\end{equation}
for any sequence $B_i\in\mathcal{B}$ of disjoint sets whose union happens to be in $\mathcal{B}$. Then $\mu$ extends to a measure on a $\sigma$-algebra containing $\mathcal{B}$. Furthermore, if $\mu$ is $\sigma$-finite on $\mathcal{B}$, then this extension is unique.
\end{theorem}

\begin{remark}
Carath\'{e}odory's theorem is proven by extending $\mu$ to an outer measure $\mu^*$ on $X$. Taking $\mathscr{F}$ to be the collection of subsets $A\subseteq X$ satisfying
\[\mu^*(E)=\mu^*(E\cap A)+\mu^*(E\setminus A)\]
for every $E\subseteq X$,
 it is shown that $\mathcal{B}\subseteq \mathscr{F}$, and that $\mathscr{F}$ is the maximal $\sigma$-algebra  to which $\mu^*$ restricts as a measure. If we take $\mathscr{F}$ to be this maximal $\sigma$-algebra (rather than some smaller $\sigma$-algebra containing $\mathcal{B}$), then the construction is unique.
\end{remark}

\begin{lemma}
The function $\mu_v$ extends uniquely to a measure on some $\sigma$-algebra $\mathscr{F}_v$ containing $\mathscr{F}_{0, v}$.
\end{lemma}

\begin{proof}
We need only show that $\mathscr{F}_{0, v}$ and $\mu$ meet the conditions of Theorem~\ref{th:cara}, which really comes down to establishing \eqref{eq:sigmaadd}. The facts that $\mu_v(\emptyset)=0$ and $\mu_v(M_{k, v})=1$ are immediate, while the (finite) additivity of $\mu_v$ follows immediately from \eqref{eq:coherent}. The $\sigma$-additivity is the main content of the lemma.

Note that, if we take the sets $A_{w, L}$ as the basis for a topology on $M_{k, v}$, then this space is naturally the inverse limit of the spaces $M_{L, v}$ with the discrete topology, for $L/K$ finite. Since these spaces are compact and Hausdorff, so is $M_{k, v}$. Now, suppose we have $B_i\in \mathscr{F}_{0, v}$. Note that the elements of $\mathscr{F}_{0, v}$ are both open and closed, and so if $\bigcup_{i\geq 0} B_i\in\mathscr{F}_{0, v}$, there must be a finite subcollection of the $B_i$ with the same union. But these sets are disjoint, and hence all but finitely many are empty. Now $\sigma$-additivity follows from finite additivity.
\end{proof}

We now show quickly that one may construct a disjoint union of a collection of measure spaces.
\begin{lemma}\label{lem:disjointunion}
Let $X_i$ be disjoint sets, for $i\in I$, let $\mathscr{F}_i$ be a $\sigma$-algebra on $X_i$, and let $\mu_i$ be a measure on $\mathscr{F}_i$. Then there is a $\sigma$-algebra $\mathscr{F}$ on $X=\bigcup X_i$ and a measure $\mu$ on $\mathscr{F}$ such that, for each $i$, $\mathscr{F}_i\subseteq \mathscr{F}$ and $\mu_i$ is the restriction to $\mathscr{F}_i$ of $\mu$.
\end{lemma}

\begin{proof}
Let $\mathscr{F}$ be the collection of subsets $B\subseteq X$ with the property that $B\cap X_i\in \mathscr{F}_i$ for all $i$. It is easy to check that $\mathscr{F}$ is a $\sigma$-algebra, and of course $\mathscr{F}_i\subseteq \mathscr{F}$ for all $i$. We can define $\mu$ on $\mathscr{F}$ by
\[\mu(B)=\sum_{i\in I}\mu_i(B\cap X_i).\]
Checking that this is a measure involves interchanging two sums, both of non-negative real numbers.
\end{proof}

Applying the construction given by Lemma~\ref{lem:disjointunion} to produce from the probability spaces $(M_{k, v}, \mathscr{F}_v, \mu_v)$ a measure space $(M_{k}, \mathscr{F}, \mu)$, we have proven the theorem. Note that, if $M_K$ is countable (as is the case when $K$ is a number field, or the function field of a curve over a countable field), then this measure space is a countable disjoint union of probability spaces (the spaces $(M_{k, v}, \mathscr{F}_v, \mu_v)$ each have total mass 1), and hence is $\sigma$-finite.
\end{proof}

Given the notion of a measured set of absolute values, and the existence of a natural such structure on $\overline{K}$ for any number field $K$, we formulate an appropriate notion of local heights on certain sorts of schemes. Our notion is motivated by the construction via N\'{e}ron divisors \cite[Chapter 10]{lang}, but we introduce a class of objects which do not quite restrict to the notion of a N\'{e}ron divisor on a variety. It would be possible to complete much of this construction with a more direct generalization of N\'{e}ron divisors, but the measure-theoretic context makes another construction more natural.

Given a measured set $M_k$ of absolute values on $k$, the appropriate analogue of an $M_K$-constant in the sense of \cite{lang} is an element of $L^1(M_k)$. Note that such functions are not precisely $M_K$-constants, even on $M_\QQ$ with the counting measure. In particular, if we enumerate the places of $\QQ$ as $v_1$ the archimedean place, $v_2$ the $2$-adic place, $v_3$ the 3-adic place, \emph{et cetera}, then $\gamma(v_n)=2^{-n}\in L^1(M_\QQ)$. But $\gamma$ is not finitely supported, and hence is not an $M_\QQ$-constant in the sense of \cite{lang}.
On the other hand, if $K$ is a number field and $\gamma$ is an $M_K$-constant in the sense of \cite{lang}, then $\gamma\in L^1(M_K)$ simply because $\gamma$ is finitely supported.

As usual, a \emph{Cartier divisor} on a $k$-scheme $X$ is represented by a cover $U_i$, for $i\in I$, by affine open sets, along with a collection of rational functions $f_i\in k(U_i)$, subject to the condition that $f_i/f_j$ and $f_j/f_i$ are both regular on $U_i\cap U_j$, for any choice of indices. As is conventional, we take a Cartier divisor to be a maximal such data, with the Cartier divisors represented by $(U_i, f_i)$ and $(U_i, g_i)$ identified if and only if $f_i/g_i$ and $g_i/f_i$ are regular on each $U_i$. The zero divisor will then be defined by $f_U=1$ for all $U\subseteq X$. The \emph{support} $\operatorname{Supp}(D)$ is defined by the condition that $P\in U\subseteq X$ is in the support of $D$ if and only if $f_U$ or $1/f_U$ is not regular at $P$.

We now define an \emph{augmented divisor} to be a Cartier divisor $(U_i, f_i)$ with an additional piece of information $\alpha_i:U_i(k)\times M_k\to \RR$ for each $i\in I$ such that
\begin{enumerate}
\item $\alpha_i(\cdot, v):U_i(k)\to \RR$ is $v$-adically continuous for each $v\in M_k$,
\item $\alpha_i(P, \cdot)$ is measurable for each $P\in U_i(k)$,
\item $|\alpha_i(P, v)|\leq \gamma(v)$ for some $\gamma\in L^1(M_k)$ and all $P\in U_i(k), v\in M_k$,
\item\label{it:welldef} $\log|f_i(P)/f_j(P)|_v=\alpha_i(P, v)-\alpha_j(P, v)$ on $U_i\cap U_j$.
\end{enumerate}
Note that, if $k$ is a number field equipped with the above weighted counting measure, then every N\'{e}ron divisor in the sense of \cite[Chapter~10]{lang} is an augmented divisor. We shall abuse notation and use the same symbol for an augmented divisor as for its underlying Cartier divisor.

To each augmented divisor $D=(U_i, f_i, \alpha_i)$, we associate a \emph{local height} (i.e., a Weil function)
\[\lambda_{X, D}:(X(k)\setminus\operatorname{Supp}(D))\times M_k\to \RR\]
by
\[\lambda_{X, D}(P, v)=-\log|f_i(P)|_v+\alpha_i(P, v)\]
for any $P\in U_i$. That this is well-defined follows from item~\eqref{it:welldef} in the definition.

Given a scheme $X$ over an algebraically closed field $k$ with a measured set of absolute values, we will write $\cadiv(X)$ for the set of Cartier divisors, and $\augdiv(X)$ for the set of augmented divisors. It is clear that both are abelian groups in a natural way, and that there is a homomorphism
\begin{equation}\label{eq:aug}\augdiv(X)\to \cadiv(X)\end{equation}
obtained by forgetting the functions $\alpha_i$. As remarked, in the case where $X$ is a projective variety we have $\nediv(X)\subseteq \augdiv(X)$ in a natural way, where $\nediv(X)$ denotes the group of N\'{e}ron divisors, and the fact that $\nediv(X)\to\cadiv(X)$ is surjective  is more-or-less the content of the local height machine. It is then of great interest to determine when the corresponding map \eqref{eq:aug} is surjective.

In order to give an answer to this question in the relevant case, we will say that a $k$-scheme $\mathscr{X}$ is a \emph{provariety} if and only if it can be exhibited as the limit of an inverse system of projective varieties over $k$ with finite transition maps.
\begin{lemma}\label{lem:cartierarepullback}
Let $\mathscr{X}$ be a provariety over $k$. Then every Cartier divisor on $\mathscr{X}$ is the pull-back of a Cartier divisor through some morphism $f:\mathscr{X}\to X$ to a projective variety. As a consequence, the homomorphism $\augdiv(\mathscr{X})\to \cadiv(\mathscr{X})$ is surjective.
\end{lemma}

\begin{proof}
Suppose that we are given a Cartier divisor $\mathscr{D}$ on $\mathscr{X}$ by the data $(U_t, g_t)$, with indices $t\in T$.
Let $I$ be a partially ordered set, and let $(X_i, f_{ij})$ be an inverse system of varieties over $k$ such that $\mathscr{X}=\varprojlim X_i$. We have canonical (affine) maps $f_i:\mathscr{X}\to X_i$, for each $i\in I$. By \cite[Chapter~31, Lemma~3.8]{stacks-project} each  $U_t$ can be expressed as $U_t=f_{i_t}^{-1}(V_{i_t})$ for some $i_t\in I$ and some affine open $V_{i_t}\subseteq X_{i_t}$. If we pare down the $U_t$ to a finite cover, all of these maps factor through some particular index $0\in I$, and so we may write  $U_t=f_{0}^{-1}(V_{t})$ for some affine opens $V_{t}\subseteq X_{0}$.

Now, since the sheaf of rational functions on $\mathscr{X}$ is the colimit of the sheaves of rational functions on the $X_i$ (essentially by definition \cite[Chapter~31, Lemma~2.2]{stacks-project}), each $g_t\in k(U_t)$ must be the pull-back to $U_t$ of some rational function $h_{i_t}\in k(f_{i_t0}^{-1}(U_t))$. Again, now that the index set $T$ is finite, we can assume without loss of generality that $i_t=0$ for all $t$, and so $g_t\in k(U_t)$ is the pull-back of some $h_t\in k(V_t)$. But now we have some data $(V_t, h_t)$ representing a Cartier divisor $D$ on some $X_0$, and by construction $\mathscr{D}=f_0^*(D)$.

Since $X_0$ is a projective variety, the Cartier divisor $D$ on $X_0$ admits the structure of a N\'{e}ron divisor \cite[Theorem~3.5, p.~261]{lang}. We may simply pull back this additional structure to $\mathscr{D}$. If $\beta_t:V_t\times M_k\to \RR$ are the functions making $D$ a N\'{e}ron divisor, it is easy to check that $\alpha_{t}(P, v)=\beta_t(f(P), v)$ gives a  collection of functions $\alpha_t$ giving $\mathscr{D}$ the structure of an augmented divisor.
\end{proof}

\begin{remark}
For the purposes of this section, it would be sufficient to define a N\'{e}ron divisor on $\mathscr{X}$ to be one that is the pull-back of a N\'{e}ron divisor on some projective variety $\mathscr{X}\to X$. This approach seems less natural given the measure-theoretic setting, but also causes significant problems in the application of Lemma~\ref{lem:tatething} below.
\end{remark}

\begin{lemma}\label{lem:heightsforzero}
Let $\mathscr{X}$ be a provariety over $k$.
If $\mathscr{D}$ is an augmented divisor whose corresponding Cartier divisor is trivial, then $\lambda_{\mathscr{X}, \mathscr{D}}(P, \cdot)\in L^1(M_k)$ for each $P\in \mathscr{X}\setminus \operatorname{Supp}(D)$, and the value of the integral
\[\int_{M_k}\lambda_{\mathscr{X}, \mathscr{D}}(P, v)d\mu(v)\]
is bounded independent of $P$.
\end{lemma}

\begin{proof}
Let $U$ be an affine open on which $\mathscr{D}$ is given by the data $(U, f, \alpha)$.
 Note that since $\mathscr{D}$ is trivial, both $f$ and $1/f$ are  regular on $U$. As in the proof of the previous lemma, we may choose a projective variety $X$, an affine open $V$, a rational function $h\in k(V)$, and a morphism $g:\mathscr{X}\to X$ such that $U=g^{-1}(V)$, and such that  $f=h\circ g$. Now, $\log|f(P)|_v=\log|h(g(P))|_v$, and so if $X$, $V$, $h$, and $g(P)$ are defined over the finite extension $L/K$ then this function is constant on sets of the form $A_{w, L}$. Since it also vanishes off of a finite collection of sets of this form, it is $L^1$. The function $\alpha(P, \cdot)$ is measurable and bounded by an $L^1$ function, and hence is $L^1$. Now $\lambda_{X, D}(P, \cdot)$ is the sum of two $L^1$ functions, and hence is itself $L^1$. Similarly, the fact that $\lambda_{X, D}(\cdot, v)$ is $v$-adically continuous on $U$ follows from the same fact for $\log|f(\cdot)|_v$ and $\alpha(\cdot, v)$.

Now, since $f$ is regular at $P$, the product formula gives
\[\int_{M_k}\log|f(P)|_vd\mu(v)=0.\]
It follows that,
\begin{eqnarray*}
\left|\int_{M_k}\lambda_{\mathscr{X}, D}(P, \cdot)d\mu(v)\right|&\leq&\left|\int_{M_k}-\log|f(P)|_vd\mu(v)\right|+\left|\int_{M_k}\alpha(P, v)d\mu(v)\right|\\
&\leq &\int_{M_k}\left|\alpha(P, v)\right|d\mu(v)\\
&\leq &\int_{M_k}\gamma(v)d\mu(v),
\end{eqnarray*}
where $\gamma$ is the $L^1$ function ensured by the definition. This bound is independent of $P$, and hence of $U$.
\end{proof}

In light of Lemma~\ref{lem:cartierarepullback}, for any Cartier divisor $\mathscr{D}$ on $\mathscr{X}$, we may define
\[h_{\mathscr{X}, \mathscr{D}}(P)=h_{X, D}(f(P))\]
for any projective variety $f:\mathscr{X}\to X$ and Cartier divisor $D$ such that $\mathscr{D}=f^* D$. The proof of Lemma~\ref{lem:cartierarepullback} ensures that any divisor on $\mathscr{X}$ can be so exhibited, and the usual functoriality properties of the Weil height \cite[Theorem~B.3.2, p.~184]{hind-silv} show that the definition is independent of the choice of $X$, at least up to a bounded error term. Indeed, all of the basic properties one would like about heights on provarieties follow quickly from the corresponding properties for varieties.

\begin{remark}
The observation that heights on provarieties have essentially the same properties as heights on projective varieties gives another approach to the construction of the canonical height $\hat{h}_{X, C, L}$. Note that the relation $y^*L\sim \alpha x^*L$, in concert with the commutative diagram in Theorem~\ref{prop:paths} implies
\[\sigma^*\mathscr{L}=\sigma^*\pi^*L = \epsilon^*y^* L  \sim \epsilon^*\left(\alpha x^*L\right)=\alpha\pi^*L=\alpha \mathscr{L}.\]
 In particular, one has a morphism of $k$-schemes $\sigma:\path\to\path$ satisfying
\[h_{\path, \mathscr{L}}(\sigma(P))=\alpha h_{\path, \mathscr{L}}(P)+O(1),\]
and one can dispatch with the correspondence entirely, and use the (suddenly applicable) arguments in \cite{callsilv} to construct the canonical height. 

We have avoided this approach both because it takes the emphasis off of the underlying correspondence, and also because the machinery required by the arguments in Section~\ref{sec:globcan} is far more elementary.
\end{remark}

\begin{theorem}\label{th:localtoglobalweilheight}
Let $\mathscr{X}$ be a provariety over $k$, and let $\mathscr{D}$ be an augmented divisor on $\mathscr{X}$. Then we have
\[h_{\mathscr{X}, \mathscr{D}}(P)=\int_{M_k}\lambda_{\mathscr{X}, \mathscr{D}}(P, v)d\mu(v)+O(1),\]
for all $P\not\in\operatorname{Supp}(\mathscr{D})$.
\end{theorem}

\begin{proof}
As shown in the proof of Lemma~\ref{lem:cartierarepullback}, there exists a projective variety $f:\mathscr{X}\to X$ and a divisor $D$ on $X$ such that $\mathscr{D}=f^*D$. If $\lambda_{X, D}$ is the local height associated to some N\'{e}ron divisor on $D$, and if $X$, $f(P)$, and $D$ are defined over the finite extension $L/K$, then we have
\begin{eqnarray*}
\int_{M_k}\lambda_{X, D}(f(P), v)d\mu(v)&=&\sum_{v\in M_L}\frac{[L_v:K_v]}{[L:K]}\lambda_{X, D}(f(P), v)\\
&=&h_{X, D}(f(P)),
\end{eqnarray*}
since the function $\lambda_{X, D}(f(P), \cdot)$ is constant on sets of the form $A_{v, L}$ for $v\in M_L$, in the notation of Theorem~\ref{th:measuredplaces}. On the other hand, we have $h_{\mathscr{X}, \mathscr{D}}(P)=h_{X, D}(f(P))+O(1)$ by definition. It thus suffices to show that
\[\int_{M_k}\left(\lambda_{\mathscr{X}, \mathscr{D}}(P, v)-\lambda_{X, D}(f(P), v)\right)d\mu(v)\]
is bounded independent of $P$. But this follows from Lemma~\ref{lem:heightsforzero}. Specifically,  $(P, v)\mapsto \lambda_{X, D}(f(P), \cdot)$ is a local height on $\mathscr{X}$ relative to $\mathscr{D}=f^*D$, and so the difference $(P, v)\mapsto\lambda_{\mathscr{X}, \mathscr{D}}(P, v)-\lambda_{X, D}(f(P), v)$ is a local height on $\mathscr{X}$ relative to some augmented divisor $\mathscr{E}$ whose underlying Cartier divisor is trivial (simply because the forgetful map $\augdiv(X)\to\cadiv(X)$ is a homomorphism).
\end{proof}


\section{Local canonical heights}

In this section we construct local canonical heights for a polarized correspondence defined over the algebraic closure $k$ of a number field. We equip $M_{k}$ with the structure of a $\sigma$-finite measure space, with measure $\mu$, as in the previous section.
In order to state the theorem, let $X$ be a projective variety over $k$, and let $C$ be a correspondence on $X$ polarized by $L$ and $\alpha>1$. Construct $\path$ as in Section~\ref{sec:geom}, and to simplify notation let $\mathscr{L}=\pi^*L$.

\begin{theorem}
Given the data above, there exists a function
\[\hat{\lambda}_{X, C, L}:\left(\path(k)\setminus\operatorname{Supp}(\mathscr{L})\right)\times M_{k}\to \RR\]
such that the following hold:
\begin{enumerate}
\item The function $\hat{\lambda}_{X, C, L}$ is the local height function associated to some augmented divisor on $\path$ with Cartier divisor $\mathscr{L}$.
\item For some rational function $f\in k(C)^*\otimes \RR$, for all $P\in\path(k)$ with  $\epsilon(P)$ not a zero or pole of $f$, and all $v\in M_k$,
\[\hat{\lambda}_{X, C, L}(\sigma(P), v)=\alpha\hat{\lambda}_{X, C, L}(P, v)-\log|f\circ\epsilon(P)|_v.\]
\item For all $P\in\path(k)$, \[\hat{h}_{X, C, L}(P)=\int_{M_{k}}\hat{\lambda}_{X, C, L}(P, v)d\mu(v).\]
\end{enumerate}
\end{theorem}

The proof requires a lemma extending an argument dating back to Tate, present also in \cite{callsilv}.
\begin{lemma}\label{lem:tatething}
Fix a set $\mathcal{X}$, a measure space $\mathcal{Y}$, and a real number $\alpha>1$. Suppose also that we are given a topology on each fibre of $\mathcal{X}\times \mathcal{Y}$, and a function $\phi:\mathcal{X}\to\mathcal{X}$ such that the induced function from $\mathcal{X}\times \mathcal{Y}$ to itself is continuous on each fibre.
Finally, suppose we are given a function $\gamma:\mathcal{X}\times\mathcal{Y}\to\RR$ which
\begin{enumerate}
\item is bounded;
\item is continuous on fibres of $\mathcal{X}\times \mathcal{Y}\to\mathcal{Y}$;
\item is measurable on fibres of $\mathcal{X}\times \mathcal{Y}\to\mathcal{X}$; and
\item vanishes off of $\mathcal{X}\times W$, for some set $W\subseteq \mathcal{Y}$ of finite measure.
\end{enumerate}
Then there exists a unique bounded function $\hat{\gamma}:\mathcal{X}\times\mathcal{Y}\to\RR$, continuous in the first coordinate and integrable in the second, with
\begin{equation}\label{eq:gammahat}\gamma(x, y)=\hat{\gamma}(\phi (x), y)-\alpha\hat{\gamma}(x, y)\end{equation}
for all $x\in\mathcal{X}$ and $y\in \mathcal{Y}$.
\end{lemma}

\begin{proof}
The construction is based on \cite[Lemma~2.2]{callsilv}. First, let $\mathcal{Z}$ be the set of functions $F:\mathcal{X}\times\mathcal{Y}\to \RR$ which are bounded,   continuous on fibres of $\mathcal{X}\times\mathcal{Y}\to\mathcal{Y}$, and measurable fibres of $\mathcal{X}\times\mathcal{Y}\to\mathcal{X}$. Note that $\gamma\in \mathcal{Z}$ by hypothesis, and the supremum norm gives $\mathcal{Z}$ the structure of a Banach space.

As in \cite[Proof of Lemma~2.2]{callsilv}, we let $\hat{\gamma}\in\mathcal{Z}$ denote the unique fixed point of the contraction $\delta\mapsto S(\delta)$ on $\mathcal{Z}$ given by \[S(\delta)(x, y)=\frac{1}{\alpha}\left(\delta( \phi (x), y)-\gamma(x, y)\right).\] The relation~\eqref{eq:gammahat} holds by definition. It remains to show that $\hat{\gamma}$ is in fact integrable (not merely measurable) in the second coordinate.

Choose $W\subseteq\mathcal{Y}$ of finite measure such that $\gamma$ vanishes off of $\mathcal{X}\times W$. Then the same is true for $S^n(\gamma)$, for every $n$, and hence for $\hat{\gamma}$. Any function which is bounded, measurable, and supported on a set of finite measure is integrable.
\end{proof}

\begin{proof}[Proof of Theorem~\ref{th:localheights}]
Let $\lambda_{X, L}$ be some local height function, given by a N\'{e}ron divisor, and let $f\in k(C)^*\otimes\RR$ be a rational function with
\[\operatorname{div}(f)=y^*L-\alpha x^*L.\]
Note that, by the basic properties of local heights \cite[Theorem~B.8.1, p.~239]{hind-silv} we have  
\begin{eqnarray*}
\lambda_{X, L}(\pi\circ\sigma(P), v)&=&\lambda_{X, L}(y\circ\epsilon(P), v)\\
&=&\lambda_{C, y^*L}(\epsilon(P), v)+O(1)\\
&=&\lambda_{C, \alpha x^*L+\operatorname{div}(f)}(\epsilon(P), v)+O(1)\\
&=&\lambda_{C, \alpha x^*L}(\epsilon(P), v)+\lambda_{C, \operatorname{div}(f)}(\epsilon(P), v)+O(1)\\
&=&\alpha\lambda_{X, L}(x\circ\epsilon(P), v)+\lambda_{C, \operatorname{div}(f)}(\epsilon(P), v)+O(1)\\
&=&\alpha\lambda_{X, L}(\pi(P), v)-\log|f\circ\epsilon(P)|_v+O(1),
\end{eqnarray*}
for any $P\in\path(k)$ so long as $\pi(P), \pi\circ\sigma(P)$ are not in the support of $L$, and $f$ and $1/f$ are regular at $\epsilon(P)$. Note that the bound implied by the $O(1)$ notation is independent of $P$.  Also, since we are applying \cite[Theorem~B.8.1, p.~239]{hind-silv} in some number field $L$ over which $X$, $C$, $L$, $f$ are all defined, we may take the $O(1)$ bound to be constant on sets of the form $A_{w, L}\subseteq M_k$, and to vanish on all but finitely many of those. In other words, if we set
\[\gamma(P, v)=\lambda_{X, L}(\pi\circ\sigma(P), v)-\alpha\lambda_{X, L}(\pi(P), v)+\log|f\circ\epsilon(P)|_v,\]
then we have just shown $\gamma$ to be a bounded, measurable function which vanishes off $(\path(k)\setminus\operatorname{Supp}(\mathscr{L})\cup\operatorname{Supp}(\sigma^*\mathscr{L}))\times W$, for some set $W$ of finite measure. We wish to show that $\gamma$ extends to a function on $\path(k)\times M_k$ which meets the conditions of Lemma~\ref{lem:tatething}.

Let $U\subseteq X$ be an affine open on which $L$ is defined locally by $g=0$, and let $\lambda_{X, L}$ be defined by
\[\lambda_{X, L}(P, v)=-\log|g(P)|_v+\beta(P, v).\] 
Then, where defined,
\begin{eqnarray}
\gamma(P, v)&=&-\log|g\circ\pi\circ\sigma(P)|_v+\beta(\pi\circ\sigma(P), v)+\alpha\log|g\circ\pi(P)|_v\nonumber\\
&&-\alpha\beta(\pi(P), v)-\log|f\circ\epsilon(P)|_v\nonumber\\
&=&\log\left|h(P)\right|_v+\beta(\pi\circ\sigma(P), v)-\alpha\beta(\pi(P), v)\label{eq:gammaisgood}
\end{eqnarray}
for
\[h=\frac{\left(g\circ\pi\right)^\alpha}{\left(f\circ\epsilon\right)\left(g\circ\pi\circ\sigma\right)}.\]
But by hypothesis, $h$ and $1/h$ are regular on $U$ (since this function defines the trivial Cartier divisor), and so \eqref{eq:gammaisgood} gives an extension of $\gamma$ to a bounded function on all of $U\times M_k$, continuous in the first variable and measurable in the second. Covering $\path$ by finitely many affine opens allows us to so extend $\gamma$ to all of $\path(k)\times M_k$.

We may now apply Lemma~\ref{lem:tatething} to $\gamma$ to produce a unique bounded function $\hat{\gamma}:\path(k)\times M_k\to \RR$, continuous in the first variable and integrable in the second, such that \[\gamma(P, v)=\hat{\gamma}(\sigma(P), v)-\alpha\hat{\gamma}(P, v).\]

Given this, we set
\[\hat{\lambda}_{X, C, L}(P, v)=\lambda_{X, L}(\pi(P), v)-\hat{\gamma}(P, v).\]
Immediately from the definition we have
\begin{eqnarray*}
\hat{\lambda}_{X, C, L}(\sigma(P), v)&=&\lambda_{X, L}(\pi\circ\sigma(P), v)-\hat{\gamma}(\sigma(P), v)\\
&=&\alpha\lambda_{X, L}(\pi(P), v)-\log|f\circ\epsilon(P))|_v+\gamma(P, v)-\hat{\gamma}(\sigma(P), v)\\
&=&\alpha\lambda_{X, L}(\pi(P), v)-\log|f\circ\epsilon(P))|_v-\alpha\hat{\gamma}(P, v)\\
&=&\alpha\hat{\lambda}_{x, C, L}(P, v)-\log|f\circ\epsilon(P))|_v
\end{eqnarray*}
wherever both sides are defined. We also note that, since $\hat{\gamma}$ is continuous in the first coordinate, measurable in the second, and bounded by an $L^1$ function of the second coordinate, the function $\hat{\lambda}_{X, C, L}$ is the local height associated to some augmented divisor with underlying Cartier divisor $L$.

We now aim to prove that these local canonical heights indeed sum to the global canonical height, following the same argument as in \cite[Theorem~2.3]{callsilv}.
We may define
\begin{equation}\label{eq:hstar}h^*(P)=\int_{M_{k}} \hat{\lambda}_{X, C, L}(P, v)d\mu(v)\end{equation}
for $P\not\in \operatorname{Supp}(\mathscr{L})$. We wish to show that this function extends to $\path(k)$.

Note that, if $g\in k(X)$, then $L'=L-\operatorname{div}(g)$ also offers a polarization of the correspondence in question. If we take as a local height $\lambda_{X, L-\operatorname{div}(g)}$ the function
\[\lambda_{X, L'}(P, v)=\lambda_{X, L}(P, v)+\log|g(P)|_v,\]
then we see that
\[y^*L'-\alpha x^*L'=y^*L-\alpha x^*L-y^*\operatorname{div}(g)+\alpha x^*\operatorname{div}(g)=\operatorname{div}(f'),\]
for $f'=(g\circ x)^\alpha f/g\circ y$. 
From this, note that defining $\gamma$ as above relative to this divisor gives
\begin{eqnarray*}
\gamma'(P, v)&=&
\lambda_{X, L'}(\pi\circ\sigma(P), v)-\alpha\lambda_{X, L'}(\pi(P), v)+\log|f'\circ\epsilon(P)|_v\\
&=&\lambda_{X, L}(\pi\circ\sigma(P), v)+\log|g\circ\pi\circ\sigma(P)|v-\alpha\lambda_{X, L}(\pi(P), v)\\
&&-\alpha\log|g\circ\pi(P)|_v+\log\left|\frac{(f\circ \epsilon(P))(g\circ x\circ \epsilon(P))^\alpha}{(g\circ y\circ \epsilon(P))}\right|_v\\
&=&\lambda_{X, L}(\pi\circ\sigma(P), v)-\alpha\lambda_{X, L}(\pi(P), v)+\log|f\circ\epsilon(P)|_v\\
&=&\gamma(P, v),
\end{eqnarray*}
Since $\pi\circ\sigma = y\circ\epsilon$ and $\pi=x\circ\epsilon$. In other words, the function $\gamma$ constructed above depends only on the linear equivalence class of $L$, and hence so does the function $\hat{\gamma}$. This gives
\[\hat{\lambda}_{X, C, L'}(P, v)=\hat{\lambda}_{X, C, L}(P, v)+\log|g(P)|_v,\]
wherever both sides are defined,  and hence by the product formula
\[\int_{M_{k}} \hat{\lambda}_{X, C, L}(P, v)d\mu(v)=\int_{M_{k}} \hat{\lambda}_{X, C, L'}(P, v)d\mu(v)\]
(for any $P\not\in \operatorname{Supp}(\mathscr{L})\cup \operatorname{Supp}(\mathscr{L}')$).

 By a standard argument, for any given $P\in\path(k)$ we can find some $L'\sim L$ such that $\pi(P)\not\in \operatorname{Supp}(L')$, and hence \eqref{eq:hstar} extends to a well-defined function $h^*:\path(k)\to \RR$.
Now, since $\hat{\lambda}_{X, C, L}$ is in particular a local height on $\path$ relative to $\mathscr{L}$, we have by Theorem~\ref{th:localtoglobalweilheight} and Theorem~\ref{th:canheight} that
\begin{equation}\label{eq:hstartclosetocan}h^*(P)=\int_{M_k}\hat{\lambda}_{X, C, L}(P, v)d\mu(v)=h_{\path, \mathscr{L}}(P)+O(1)=\hat{h}_{X, C, L}(P)+O(1).\end{equation}
But note that, for
\begin{equation}\label{eq:pnotinsupp}P\not\in\operatorname{Supp}(\mathscr{L})\cup\operatorname{Supp}(\sigma^*\mathscr{L})\end{equation} we have
\begin{eqnarray*}
h^*(\sigma(P))&=&\int_{M_k}\hat{\lambda}_{X, C, L}(\sigma(P), v)d\mu(v)\\
&=&\int_{M_k}\alpha\hat{\lambda}_{X, C, L}(P, v)d\mu(v)-\int_{M_k}\log|f_i\circ\epsilon(P)|_vd\mu(v)\\
&=&\alpha h^*(P).
\end{eqnarray*}
As usual, we can always replace $\mathscr{L}$ by a linearly equivalent divisor so that all terms are defined at $P$, and hence we have $h^*(\sigma(P))=\alpha h^*(P)$ for all $P\in \path(k)$. This in turn yields
\[\left|h^*(P)-\hat{h}_{X, C, L}(P)\right|=\alpha^{-n}\left|h^*(\sigma^n(P))-\hat{h}_{X, C, L}(\sigma^n(P))\right|=O(\alpha^{-n}),\]
where the implied constant depends only on $X$, $C$, and $L$. Letting $n\to\infty$, we have $h^*=\hat{h}_{X, C, L}$, as claimed. 
\end{proof}

\end{document}